%% file: paper.tex
\documentclass[a4paper,10pt]{article} 

\usepackage[%
  layoutsize={152.4mm,228.6mm},%
  layoutoffset=20mm,%
  twoside,%
  nomarginpar,%
  inner=25mm,%
  outer=10mm,%
  vmargin=17mm,%
  head=5mm,%
  headsep=2mm,%
  foot=7mm,%
]{geometry}
\usepackage[T1]{fontenc}
\usepackage{mathptmx}

\usepackage[compact]{titlesec}

\usepackage{fancyhdr}
\fancypagestyle{plain}{%
  \fancyhf{}
  \fancyfoot[CO,CE]{\thepage}
  \renewcommand{\headrulewidth}{0mm}
}

\usepackage{amsmath,amssymb,amsthm}
\usepackage{graphicx}
\usepackage{url}

\theoremstyle{plain}
\newtheorem{theorem}{Theorem}
\newtheorem{corollary}[theorem]{Corollary}
\newtheorem{proposition}[theorem]{Proposition}
\newtheorem{lemma}[theorem]{Lemma}

\theoremstyle{definition}
\newtheorem{definition}[theorem]{Definition}

\newtheorem{example}[theorem]{Example}

\theoremstyle{remark}
\newtheorem{remark}[theorem]{Remark}

\usepackage[colorlinks]{hyperref}

\newcommand{\numberofauthors}{}
\newcommand{\authorAfullname}{}
\newcommand{\authorAindexname}{}
\newcommand{\authorAaddress}{}

\newcommand{\authorBfullname}{}

\newcommand{\authorBaddress}{}

\newcommand{\authorCfullname}{}

\newcommand{\authorCaddress}{}
\newcommand{\authorDfullname}{}

\newcommand{\authorDaddress}{}

\newcommand{\papertitle}{}

\newcommand{\processpaperdata}{%
  \setcounter{section}{0}
  \renewcommand{\leftmark}{%
    \ifcase\numberofauthors
      {}\or
      {\authorAfullname}\else
      {\authorAfullname{} et al.}
    \fi
  }
  \vspace*{3pt}
  \begin{center}{\bfseries\huge\papertitle}\end{center}
  \ifcase\numberofauthors
    {}\or
    {\begin{center}
      \begin{minipage}{50mm}\centering\authorAfullname\\\authorAaddress\end{minipage}
    \end{center}}\or
    {\begin{center}
      \begin{minipage}{50mm}\centering\authorAfullname\\\authorAaddress\end{minipage}
      \hspace{5mm}
      \begin{minipage}{50mm}\centering\authorBfullname\\\authorBaddress\end{minipage}
    \end{center}}\or
    {\begin{center}
      \begin{minipage}{50mm}\centering\authorAfullname\\\authorAaddress\end{minipage}
      \hspace{5mm}
      \begin{minipage}{50mm}\centering\authorBfullname\\\authorBaddress\end{minipage}
     \end{center}

     \begin{center}
      \begin{minipage}{50mm}\centering\authorCfullname\\\authorCaddress\end{minipage}
     \end{center}}\or
    {\begin{center}
      \begin{minipage}{50mm}\centering\authorAfullname\\\authorAaddress\end{minipage}
      \hspace{5mm}
      \begin{minipage}{50mm}\centering\authorBfullname\\\authorBaddress\end{minipage}
     \end{center}

     \begin{center}
      \begin{minipage}{50mm}\centering\authorCfullname\\\authorCaddress\end{minipage}
      \hspace{5mm}
      \begin{minipage}{50mm}\centering\authorDfullname\\\authorDaddress\end{minipage}
     \end{center}}\else
    {}
  \fi  
}

\renewenvironment{abstract}{\textbf{Abstract.}}{}

\begin{document}

\pagestyle{fancy}
\fancyhf{}
\fancyhead[RO,LE]{{\slshape Festschrift in Honor of Uwe Helmke}}
\fancyhead[LO,RE]{{\nouppercase{\slshape\leftmark}}}
\fancyfoot[CO,CE]{\thepage}
\renewcommand{\headrulewidth}{0mm}

\input{content.tex}

\nocite{*}
\bibliographystyle{abbrv}
\bibliography{refs}

\end{document}

%% file: content.tex



\renewcommand{\numberofauthors}{1} 

\renewcommand{\authorAfullname}{Jonathan H. Manton}
\renewcommand{\authorAindexname}{Manton, Jonathan H.}
\renewcommand{\authorAaddress}{The University of Melbourne\\
Victoria, 3010, Australia\\
\texttt{j.manton@ieee.org}}

\renewcommand{\papertitle}{Optimisation Geometry}

\processpaperdata 

{





\newcommand{\reals}{\mathbb{R}}
\newcommand{\restrict}[2]{\left.{#1}\right|_{#2}}


\begin{abstract}
This article demonstrates how an understanding of the geometry of
a family of cost functions can be used to develop efficient numerical
algorithms for real-time optimisation.  Crucially, it is not the
geometry of the individual functions which is studied, but the
geometry of the family as a whole.  In some respects, this challenges
the conventional divide between convex and non-convex optimisation
problems because none of the cost functions in a family need be
convex in order for efficient numerical algorithms to exist for
optimising in real-time any function belonging to the family.  The
title ``Optimisation Geometry'' comes by analogy from the study of
the geometry of a family of probability distributions being called
information geometry.
\end{abstract}

\section{Introduction and Motivation}
\label{sec:intro}

Classical optimisation theory is concerned with developing
algorithms that scale well with increasing problem size and is
therefore well-suited to ``one-time'' optimisation tasks such as
encountered in the planning and design phases of an engineering
endeavour.  Techniques from classical optimisation theory are often
applied to ``real-time'' optimisation tasks in signal
processing applications, yet real-time optimisation problems have
their own exploitable characteristics.

The often overlooked perspective this article brings to real-time
optimisation problems is that the family of cost functions should
be studied as a whole.  This leads to a nascent theory of real-time
optimisation that explores the theoretical and practical consequences
of understanding the topology and geometry of how a collection of
cost functions fit together.

For the purposes of this article, real-time optimisation
is the challenge of developing a numerical algorithm
taking a parameter value $\theta \in \Theta$ as input, and
returning relatively quickly a suitable approximation to an element of
\begin{equation}
\label{eq:xast}
\left\{ x_\ast \in X \mid f(x_\ast) = \min_x f(x;\theta) \right\}
\end{equation}
where the parametrised cost functions $f(\cdot;\theta)$ are known
in advance.  Since combinatorial and other non-smooth optimisation
problems are less amenable to the methods introduced in this article,
for the moment it may be assumed that $X$ and $\Theta$ are
differentiable manifolds and $f: X \times \Theta \rightarrow
\mathbb{R}$ is a smooth function.  (An important generalisation involving
smooth fibre bundles will be introduced in Section~\ref{sec:fb}.)

An example of real-time optimisation in signal processing
is maximum-likelihood estimation, where $x$ is the parameter to be
estimated from the observation $\theta$ and $f(x;\theta)$ is the
negative logarithm of the statistical likelihood function.  In a
communications system, if the transmitted message is $x$ and the
received packet is $\theta$ then each time a new packet is received
the optimisation problem (\ref{eq:xast}) must be solved to recover
$x$ from $\theta$.

The distinguishing features setting apart real-time
optimisation from classical optimisation are: the class of cost
functions $f(\cdot;\theta)$ is known in advance; the class is
relatively small (meaning $\Theta$ is finite-dimensional); an
autonomous algorithm is required that quickly and efficiently
optimises $f(\cdot;\theta)$ for (almost) any value of $\theta$.

Real-time optimisation problems also differ from adaptive problems
in that global robustness is important.  Real-time algorithms must
be capable of handling in turn any sequence of values for the
parameter $\theta$, whereas adaptive algorithms can assume successive
values of $\theta$ will be close to each other, thereby simplifying
the problem to that of tracking perturbations.  Nevertheless, there
are similarities because it is proposed here, in essence, to solve
real-time optimisation problems by reducing them to tracking problems.
Geometry facilitates this reduction.

The recent popularity of convex optimisation methods in signal
processing exemplifies the earlier remark that classical optimisation
theory is often applied to real-time optimisation problems.  While
great benefit has come from the realisation that classes of signal
processing problems can be converted into convex optimisation
problems such as Second-Order Cone Programming problems, this
approach does not exploit the relationships between the different
cost functions in the same family.

Although convexity currently determines the dichotomy of optimisation
--- convex problems are ``easy'' and non-convex problems are
``hard''~\cite{bk:Nazareth:opt} --- this is 
irrelevant for real-time optimisation because the complexity of
real-time algorithms can be reduced by using 
the results of offline computations made during the design
phase.  An extreme example is when all the cost
functions $f(\cdot;\theta)$ are just translated versions of a cost
function $h(\cdot)$, such as $f(x;\theta) = h(x-\theta)$.  The cost
function $h$ might be difficult to optimise, but once
its minimum $x_{\ast}$ has been found, the real-time optimisation
algorithm itself is trivial: given $\theta$, the minimum of
$f(x;\theta) = h(x-\theta)$ is immediately computed to be $x_{\ast}
+ \theta$.

This line of reasoning extends to more general situations.  For
concreteness, take the parameter space $\Theta$ to be the circle
$S^1$ (or, in fact, any compact manifold).  As before, each individual
cost function $f(\cdot;\theta)$ might be difficult to optimise, but
provided the location of the minimum varies smoothly for almost
every value of $\theta$, the following (simplified) algorithm
presents itself.  Choose a finite number of parameter values
$\theta_1,\cdots,\theta_n \in \Theta$.  Using whatever means possible,
compute beforehand the minima $x_1,\cdots,x_n \in X$ of the cost
functions $f(x;\theta_i)$, that is, $f(x_i) = \min_x f(x;\theta_i)$.
The minimum of $f(\cdot;\theta)$ generally can be found quickly and
reliably by determining the $\theta_i$ closest to $\theta$, and
starting with the pair $(x_i,\theta_i)$, applying a homotopy
method~\cite{bk:Allgower:num_cont} to find the minimum of successive
cost functions $f(\cdot; \theta_i + k \epsilon)$ for $k=1,\cdots,K$,
where $\epsilon = (\theta - \theta_i)/K$; see Section~\ref{sec:halg}
for details. Thus, the overall complexity of real-time optimisation
is determined by how the cost functions $f(\cdot;\theta)$ change
as $\theta$ is varied, and not by any classical measure of the
difficulty of optimising a particular cost function in the family
$\{f(\cdot;\theta) \mid \theta \in \Theta\}$.

Another reason for believing in advance that the geometry of the
family of cost funtions as a whole will help determine the computational
complexity of real-time optimisation is that work on topological
complexity and real complexity theory has already demonstrated that
the geometry of a problem provides vital clues for its numerical
solution~\cite{Smale:top_alg, Smale:fund_thm_alg, bk:Blum:newt}.
(Another example of the efficacy of using geometry to develop
numerical solutions is~\cite{Arnold:ext_calc}.)

Shifting from a Euclidean-based perspective of optimisation to a
manifold-based perspective is expected to facilitate the development
of a complexity theory for real-time optimisation.  Moving to a
differential geometric setting accentuates the geometric aspects
while attenuating artifacts introduced by specific choices of
coordinate systems used to describe an optimisation
problem~\cite{Manton:opt_mfold,Manton:jad,bk:Jongen:optimization}.
Furthermore, a wealth of problems occur naturally on
manifolds~\cite{bk:Chirikjian:harm,Manton:opt_mfold,Manton:Cicassp_diffg},
and coaxing them into a Euclidean framework is artificial and
not necessarily beneficial.

The flat, unbounded geometry of Euclidean space places no topological
restrictions on cost functions $f: \reals^n \rightarrow \reals$.
Focusing on compact manifolds creates a richer structure for
algorithms to exploit while maintaining practical relevance: compact
Lie groups, and Grassmann and Stiefel manifolds occur in a range
of signal processing applications.  To the extent that no algorithm
can search an unbounded region in finite time, the restriction to
compact manifolds is not necessarily that restrictive.  As a first
step then, this article focuses on optimisation problems on compact
manifolds.

One way to visualise how the cost functions in a family fit together
is to imagine the mapping $\theta \mapsto
f(\cdot;\theta)$ carving out a subset of the space of all (smooth)
functions.  This is essentially the approach taken in information
geometry~\cite{bk:Amari:info_geom}, where $f(\cdot;\theta)$ is a
probability density function rather than a cost function.  It seems
appropriate then to use Optimisation Geometry as the title of this
article.

Tangentially, it is remarked that even for one-time optimisation
problems, it is not clear to the author that convexity is the
fundamental divide separating easy from hard problems.  Convexity
might be an artifact of focusing on optimisation problems on
$\reals^n$ rather than on compact manifolds.  There do not exist
any nontrivial convex functions $f: M \rightarrow \reals$ on a
compact connected manifold $M$ --- if $f$ is
convex~\cite{bk:Udriste:convex} then it is necessarily a constant
--- yet if $M$ were a circle or a sphere, presumably there are
numerous classes of cost functions that can be ``easily'' optimised.

\section{A Fibre Bundle Formulation of Optimisation}
\label{sec:fb}

A real-time optimisation algorithm computes a possibly
discontinuous mapping $g$ from $\Theta$ to $X$.  Given an input
$\theta \in \Theta$, the algorithm returns $g(\theta) \in X$ where
$g$ satisfies
\begin{equation}
f(g(\theta);\theta) = \min_x f(x;\theta)
\end{equation}
for all, or almost all, $\theta \in \Theta$.  (Randomised
algorithms are not considered here.) In a certain sense then, the
additional information contained in the cost functions $f(\cdot;\theta)$
is irrelevant; if a closed-form expression for $g$ can be determined
then the original functions $f$ can be discarded.

However, often in practice it is too hard (or not worth the effort)
to find $g$ explicitly.  Optimisation algorithms therefore typically
make use of the cost function, finding the minimum by 
moving downhill, for example.  With the caveat that there is no
need to remain with the original cost functions $f(\cdot;\theta)$
--- they can be replaced by any other family provided there is no
consequential change to the ``optimising function'' $g$ --- a first
attempt at studying the complexity of real-time optimisation problems
can be made by endeavouring to link the geometry of $f$ with the
computational complexity of evaluating the optimising function $g$.

Define $M$ to be the product manifold $M = X \times \Theta$, and
let $\pi: M \rightarrow \Theta$ denote the projection $(x,\theta)
\mapsto \theta$.  The family of cost functions $f(\cdot;\theta)$
can be thought of as a single function $f: M \rightarrow
\reals$, that is, as a scalar field on $M$.  Provided $f: M \rightarrow
\reals$ is smooth, the manifold $M$ relates to how the cost
functions fit together.

If $f$ is not smooth, a reparametrisation of the family of cost
functions could be sought to make it smooth; in essence, a
parametrisation $\theta \mapsto f(\cdot;\theta)$ is required for
which smooth perturbations of $\theta$ result in smooth perturbations
of the corresponding cost functions.  To increase the chance of
this being possible, an obvious and notationally convenient
generalisation of the real-time optimisation problem is introduced.

\begin{definition}[Fibre Bundle Optimisation Problem]
\label{def:fbop}
Let $M$ be a smooth fibre bundle over the base space $\Theta$ with
typical fibre $X$ and canonical projection $\pi: M \rightarrow
\Theta$.  Let $f: M \rightarrow \reals$ be a smooth function.  The
fibre bundle optimisation problem is to devise an algorithm computing
an \emph{optimising function} $g: \Theta \rightarrow M$ that
satisfies $(\pi \circ g)(\theta) = \theta$ and $(f \circ g)(\theta)
= \min_{p \in \pi^{-1}(\theta)} f(p)$ for all $\theta \in \Theta$.
\end{definition}

\textbf{Standing Assumptions:}
For mathematical simplicity, it is assumed throughout that 
$M$, $\Theta$ and $X$ in Definition~\ref{def:fbop} are
compact.  Smoothness means $C^{\infty}$-smoothness.

If $M = X \times \Theta$ then the only difference from before is
that the output of the algorithm is now a tuple $(x_\ast,\theta)
\in M$ rather than merely $x_\ast \in X$.  
Allowing $M$ to be a non-trivial bundle is useful in practice,
as now demonstrated.

\begin{example}
Let $M$ and $\Theta$ be compact connected manifolds. If $\pi: M
\rightarrow \Theta$ is a submersion then it is necessarily surjective
and makes $M$ a fibre bundle.  Given a smooth $f: M \rightarrow
\reals$, the fibre bundle optimisation problem is equivalent to the
constrained optimisation problem of minimising $f(p)$ subject to
$\pi(p) = \theta$.
\end{example}

\begin{example}
\label{ex:stgr}
Let $\mathrm{St}(k,n) = \{X \in \reals^{n \times k} \mid X^T X =
I\}$ denote a Stiefel manifold and $\mathrm{O}(k) = \{X \in \reals^{k
\times k} \mid X^T X = I\}$ an orthogonal group.  The Grassmann
manifold $\mathrm{Gr}(k,n)$ is a quotient space of $\mathrm{St}(k,n)$,
and in particular, $\mathrm{St}(k,n)$ decomposes as a bundle $\pi:
\mathrm{St}(k,n) \rightarrow \mathrm{Gr}(k,n)$ with typical fibre
$\mathrm{O}(k)$.  Given a smooth function $f: \mathrm{St}(k,n)
\rightarrow \reals$, the corresponding fibre bundle optimisation
problem is to minimise $f(X)$ subject to the range-space of $X$
being fixed (that is, that $\pi(X)$ is known).  A related optimisation
problem (involving a constraint on the kernel rather than the
range-space of $X$) occurs naturally in low-rank approximation
problems~\cite{Manton:wlra}.
\end{example}

The optimisation problem in Example~\ref{ex:stgr} can be written
in parametrised form by changing $f$ to $\tilde f: \mathrm{Gr}(k,n)
\times O(k) \rightarrow \reals$, but if $f$ is smooth then $\tilde
f$ need not be continuous.  Fibre bundles allow for twists in the
global geometry.

\begin{example}
Another decomposition of $\mathrm{St}(k,n)$ is $\pi: \mathrm{St}(k,n)
\rightarrow S^{n-1}$ where $\pi(X)$ is the first column of $X$.
This corresponds to interpreting an element $X \in \mathrm{St}(k,n)$
as a point in the $(k-1)$-dimensional orthogonal frame bundle of
the $(n-1)$-dimensional sphere.  More generally, fibre bundle
optimisation problems arise whenever a smooth function $f$ is defined
on a tangent bundle, sphere bundle, (orthogonal) frame bundle or
normal bundle of a manifold $M$, and it is required to optimise
$f(p)$ subject to $p$ being constrained to lie above a specified
point on $M$.
\end{example}

\begin{remark}
\label{rem:restrict}
Fibre bundle optimisation problems (Definition~\ref{def:fbop})
decompose into lower-dimensional fibre bundle optimisation problems.
If $\tilde\Theta$ is a submanifold of $\Theta$ then the restriction
of $\pi$ to $\pi^{-1}(\tilde\Theta)$ makes $M \cap \pi^{-1}(\tilde\Theta)$
into a fibre bundle over $\tilde\Theta$.  Conversely, a fibre bundle
optimisation problem can be embedded in a higher-dimensional fibre
bundle optimisation problem.
\end{remark}

The optimising function $g$ in Definition~\ref{def:fbop} would be
a section if it were smooth, but in general $g$ need not be everywhere
continuous much less smooth.  This is handled by imposing a niceness
constraint on the optimisation problem.

\begin{definition}[Niceness]
\label{def:nice}
The fibre bundle optimisation problem in Definition~\ref{def:fbop}
is deemed to be nice if there exist a finite number of connected
open sets $\Theta_i \subset \Theta$ whose union is dense in $\Theta$,
and there exist smooth local sections $g_i: \Theta_i \rightarrow
M$ such that $(f \circ g_i)(\theta) = \min_{p \in \pi^{-1}(\theta)}
f(p)$ whenever $\theta \in \Theta_i$.
\end{definition}

The requirement that the $g_i$ are sections means $\pi(g_i(\theta))
= \theta$ for every $\theta \in \Theta_i$.  The smallest number of
connected open sets required in Definition~\ref{def:nice} can be
considered to be the topological complexity of the optimisation
problem by analogy with the definition of topological complexity
in~\cite{Smale:top_alg}; note though that the $g_i$ are required
to be smooth in Definition~\ref{def:nice} whereas Smale
required only continuity.

\begin{remark}
\label{rem:theta}
A more practical definition of niceness might require the $\Theta_i$
in Definition~\ref{def:nice} to be semialgebraic sets, perhaps even
with a limit placed on the number of function evaluations
required to test if $\theta$ is in $\Theta_i$.  This is not seen
as a major issue though because it is always possible to evaluate
more than one of the $g_i$ at $\theta$ and choose the one which
gives the lowest value of $f(g_i(\theta))$;
the algorithm for computing $g_i$ can return whatever
it likes if $\theta \notin \Theta_i$. See also Section~\ref{sec:halg}.
\end{remark}

Whereas Section~\ref{sec:intro} only required a real-time optimisation
algorithm to compute the correct answer for \emph{almost} all values
of $\theta$, the standing assumption of compactness together with
restricting attention to nice problems means the algorithm can be
required to work for all $\theta$; see Remarks~\ref{rem:limit}
and~\ref{rem:closed}.

\begin{remark}
\label{rem:limit}
The compactness of $M$ means that if $\theta_n \in \Theta_i$,
$\theta_n \rightarrow \theta$ then $\{g_i(\theta_n)\}_{n=1}^{\infty}$
has at least one limit point, call it $q$.  Then $\pi(q) = \theta$
and $f(q) = \min_{p \in \pi^{-1}(\theta)} f(p)$.  Therefore, if a
fibre bundle optimisation problem is nice (Definition~\ref{def:nice})
then an optimising function exists on the whole of $\Theta$
(Definition~\ref{def:fbop}).
\end{remark}

\begin{remark}
\label{rem:graph}
In Definition~\ref{def:fbop}, the geometry of the optimisation
problem is encoded jointly by $M$ and $f$.  It is straightforward
to reduce $f$ to a canonical form by replacing $M$ with the graph
$\Gamma = \{(p,f(p)) \in M \times \reals \mid p \in M\}$.  Then $f$
becomes the height function $(x,y) \mapsto y$ and the geometry of
the optimisation problem is encoded in how $\Gamma$ sits inside $M
\times \reals$.  
\end{remark}

As a visual aid, it can be assumed, from Remark~\ref{rem:graph} and
the Whitney embedding theorem, that $M$ is embedded in Euclidean
space and the level sets $f^{-1}(c)$ are horizontal slices of $M$.

\section{The Torus}
\label{sec:torus}

To motivate subsequent developments, this section primarily considers
fibre bundle optimisation problems on the product bundle $M=S^1
\times S^1$.  The function $f: S^1 \times S^1 \rightarrow \reals$
can be thought of as defining the temperature at each point of a
torus.  Definitions and results will be stated in generality though,
for arbitrary $M$.

\subsection{Fibre-wise Morse Functions}

Minimising $f: S^1 \times S^1 \rightarrow \reals$ restricted to a
fibre is simply the problem of minimising a real-valued function
on a circle.  The smoothness of $f$ and the compactness of $S^1$
ensure the existence of at least one global minimum per fibre.

To give more structure to the set of critical points, it is common
to restrict attention either to real-analytic functions or Morse
functions.  Optimisation of real-analytic functions will not be
considered here, but may well prove profitable for the study of
gradient-like algorithms for fibre bundle optimisation problems.

If $h: S^1 \rightarrow \reals$ is Morse, meaning all its critical
points are non-degenerate, then its critical points are isolated
and hence finite in number.  Furthermore, the Newton method for
optimisation converges locally quadratically to non-degenerate
critical points.  These are desirable properties that will facilitate
the development of optimisation algorithms in Sections~\ref{sec:newt}
and~\ref{sec:halg}.

\begin{definition}[Fibre-wise Morse Function]
\label{def:fwmf}
A \emph{fibre-wise critical point} $p$ of the function $f$ in
Definition~\ref{def:fbop} is a critical point of
$\restrict{f}{\pi^{-1}(\pi(p))}$, the restriction of $f$ to the
fibre $\pi^{-1}(\pi(p))$ containing $p$.  It is non-degenerate if
the Hessian of $\restrict{f}{\pi^{-1}(\pi(p))}$ at $p$ is non-singular.
If all fibre-wise critical points of $f$ are non-degenerate then
$f$ is a fibre-wise Morse function.
\end{definition}

\begin{remark}
\label{rem:morse}
Note that $f$ being fibre-wise Morse differs from $f$ being Morse;
a non-degenerate fibre-wise critical point need not be a critical
point of $f$, and even if it were, it need not be non-degenerate
as a critical point of $f$.
\end{remark}

\begin{lemma}
\label{lem:N}
Let $f: M \rightarrow \reals$ be a fibre-wise Morse function
(Definition~\ref{def:fwmf}) on the bundle $\pi: M \rightarrow
\Theta$ (Definition~\ref{def:fbop}).  The set $N$ of fibre-wise
critical points is a submanifold of $M$ with the same
dimension as $\Theta$.  It intersects each fibre $\pi^{-1}(\theta)$
transversally.
\end{lemma}
\begin{proof}
It suffices to work locally; let $U \subset \Theta$ be open.  Denote
by $VM$ the vertical bundle of $M$; it is a subbundle of the tangent
bundle $TM$. Let $s_1,\cdots,s_k: \pi^{-1}(U) \rightarrow VM$ be a
local basis, where $k = \dim M - \dim \Theta$.
(The $s_i$ are local smooth sections of $VM$
such that $\{s_1(p),\cdots,s_k(p)\}$ is a basis for $V_pM$ for every
$p \in \pi^{-1}(U)$.)  Define $e: \pi^{-1}(U) \rightarrow \reals^k$
by $e(p) = (df(s_1(p)),\cdots,df(s_k(p)))$.  Then the set of
fibre-wise critical points is given locally by $N \cap \pi^{-1}(U)
= e^{-1}(0)$.  Fix $p \in N$.  Since $f$ is fibre-wise Morse, $de_p$
restricted to $V_pM$ is non-singular.  Therefore $de_p$ is surjective
and $\ker\, de_p + V_pM = T_pM$.  Thus, $e^{-1}(0)$ is an embedded
submanifold of $M$, it has dimension $\dim M - k = \dim \Theta$,
and it intersects each fibre transversally.
\end{proof}

The situation is especially nice on the torus: Lemma~\ref{lem:N}
implies that the set $N$ of fibre-wise critical points of a fibre-wise
Morse function is a disjoint union of a finite number of circles,
with each circle winding its way around the torus the same number
of times.  Precisely, there is an integer $b$ such that for any
$\theta$, each connected component of $N$ intersects the fibre
$\pi^{-1}(\theta) = S^1 \times \{\theta\}$ precisely $b$ times.

As soon as the fibre-wise critical points of $f$ are known at a
single fibre $\pi^{-1}(\theta)$, the fibre-wise critical points of
$f$ at another fibre $\pi^{-1}(\theta')$ can be determined 
by tracking each of the points in $N \cap \pi^{-1}(\theta)$ as
$\theta$ moves along a continuous path to $\theta'$.  This is
referred to as following the circles in $N$ from one fibre to
another.

Investing more effort beforehand can obviate the need to follow
more than one circle.  A lookup table can record the circle in $N$
on which the minimum lies based on which region contains $\theta$.
Proposition~\ref{prop:nice} formalises this.  (In practice, there
may be reasons for deciding to track more than one circle;
see Remark~\ref{rem:theta}.)

\begin{proposition}
\label{prop:nice}
If $f$ in Definition~\ref{def:fbop} is fibre-wise Morse
(Definition~\ref{def:fwmf}) then the fibre bundle optimisation
problem is nice (Definition~\ref{def:nice}).
\end{proposition}
\begin{proof}
Let $N$ be the set of fibre-wise critical points of $f$.  For $\theta
\in \Theta$, $N \cap \pi^{-1}(\theta)$ is a finite set of points
because $\pi^{-1}(\theta)$ is compact and $N \pitchfork \pi^{-1}(\theta)$
with $\dim N + \dim \pi^{-1}(\theta) = \dim M$; see Lemma~\ref{lem:N}. Therefore there
exist an open neighbourhood $U_\theta \subset \Theta$ of $\theta$
and local smooth sections $s^{(\theta)}_1,\cdots, s^{(\theta)}_{k_\theta}:
U_\theta \rightarrow M$ such that $N \cap \pi^{-1}(U_\theta) =
\cup_{i=1}^{k_\theta} s^{(\theta)}_i(U_\theta)$; pictorially, each
section traces out a distinct component of $N \cap \pi^{-1}(U_\theta)$.
Let $V_\theta \subset \Theta$ be an open neighbourhood of $\theta$
whose closure $\overline{V_\theta}$ is contained in $U_\theta$.  By
compactness there exist a finite number of the $V_\theta$ which
cover $\Theta$; denote these sets by $V_{\theta_i}$. Let $J_{ij} =
\{\theta \in \overline{V_{\theta_i}} \mid f(s^{(\theta_i)}_j(\theta))
= h(\theta)\}$ where $h(\theta) = \min_{p \in \pi^{-1}(\theta)}
f(p)$.  Each $J_{ij}$ is a closed subset of
$\overline{V_{\theta_i}}$ because $h$ is continuous.  Furthermore,
$\cup_j J_{ij} = \overline{V_{\theta_i}}$ and hence $\cup_{ij}
J_{ij} = \Theta$.  Let $\Theta_{ij}$ denote the interior of $J_{ij}$.
Since $J_{ij} \setminus \Theta_{ij}$ is nowhere dense,
$\cup_{ij} \Theta_{ij}$ is
dense in $\Theta$.  The requirements of Definition~\ref{def:nice}
are met with $g_{ij}(\theta) = s^{(\theta_i)}_j(\theta)$.
\end{proof}

\begin{remark}
\label{rem:closed}
A stronger definition of niceness could have been adopted: each
$g_i$ in Definition~\ref{def:nice} could have been required to be
a smooth optimising function on $\overline{\Theta_i}$, the closure
of $\Theta_i$.  Also, because there are only a finite number of
sets involved, $\cup_i \Theta_i$ is dense in $\Theta$ if and only
if $\cup_i \overline{\Theta_i} = \Theta$.
\end{remark}

\subsection{Connection with Morse Theory}

It is natural to ask what role Morse theory plays in
real-time optimisation.  After all, Morse theory contributes to
one-time optimisation problems by providing information about the
number, type and to some extent the location of critical points.

The short answer is the connection between Morse theory and
real-time optimisation is more subtle than for one-time optimisation.
The fibre bundle formulation of real-time optimisation highlights that
real-time optimisation is concerned with \emph{constrained}
optimisation.  It is not the level sets
$\{p \in M \mid f(p) = c\}$ that are important for real-time
optimisation but how they intersect the fibres $\pi^{-1}(\theta)$.
From an algorithmic perspective, whereas one-time optimisation
algorithms are required to find (isolated) critical points, real-time
optimisation algorithms (at least from the viewpoint of this article)
are required to track the critical points from fibre to fibre.

Nevertheless, for completeness, this section recalls what classical
Morse theory says about the torus.  Let $f: M \rightarrow
\reals$ be a smooth Morse function on $M = S^1 \times S^1$ with
distinct critical points having distinct values.  This is a mild
assumption in practice because an arbitrarily small perturbation
of $f$ can always be found to enforce this.

Morse theory explains how the level sets $f^{-1}(c)$ fit together
to form $M$.  The fibre bundle optimisation problem is to find the
smallest $c$ for which $f^{-1}(c)$ intersects the submanifold
$\pi^{-1}(\theta)$ for a given $\theta$.

If $p \in \pi^{-1}(\theta)$ is a local minimum of $f$ then it is
also a local minimum of $\restrict{f}{\pi^{-1}(\theta)}$, and
similarly for a local maximum.  In both cases, $p$ is an isolated
critical point of $\restrict{f}{\pi^{-1}(\theta)}$.  This need not
be true though if $p$ is a saddle point of $f$.

Let $p_0,\cdots,p_{n-1}$ denote the critical points of $f$ ordered
so the values $c_i = f(p_i)$ ascend.  The genus of the torus dictates
that the number of saddle points equals the total number of local
minima and maxima, therefore $n \geq 4$.

For $c \in [c_0,c_{n-1}]$ a regular value of $f$, $f^{-1}(c)$ is a
compact one-dimensional manifold and hence diffeomorphic to a finite
number of circles.  The number of circles changes by
one as $c$ passes through a critical value.  In particular,
$f^{-1}(c_0)$ is a single point, $f^{-1}(c)$ for $c \in (c_0,c_1)$
is diffeomorphic to $S^1$, and $f^{-1}(c_1)$ is either diffeomorphic
to a circle plus a distinct point, or it is diffeomorphic to two
circles joined at a single point.  In general, $f^{-1}(c_i)$ is
either diffeomorphic to zero or more copies of a circle plus a
distinct point, or it is diffeomorphic to zero or more copies of a
circle plus two circles joined at a single point.  The former occurs
when $p_i$ is a local extremum and the latter occurs when $p_i$ is
a saddle point.

Not only is $f^{-1}(c)$ diffeomorphic to a finite number of circles
for $c$ a regular value, but $\pi^{-1}(\theta)$ is also diffeomorphic
to a circle.  Visually then, increasing $c$ corresponds to sliding
one or more rubber bands along the surface of the torus, and of
interest is when one of these rubber bands first hits the circle
$\pi^{-1}(\theta)$.  The point of first contact is either a critical
point of $f$ or a non-transversal intersection point of $f^{-1}(c)
\cap \pi^{-1}(\theta)$.  Indeed, if $p \in f^{-1}(c) \cap
\pi^{-1}(\theta)$ is not a critical point of $f$ then $p$ is a
critical point of $\restrict{f}{\pi^{-1}(\theta)}$ if and only if
$p$ is a non-transversal intersection point of $f^{-1}(c)$ with
$\pi^{-1}(\theta)$.  This connects with Definition~\ref{def:fwmf}.

\section{Newton's Method and Approximate Critical Points}
\label{sec:newt}

The Newton method is the archtypal iterative algorithm for function
minimisation.  Whereas its global convergence properties are intricate
--- domains of attraction can be fractal --- the local convergence
properties of the Newton method are well understood.  The advantage
of real-time optimisation over one-time optimisation is it
suffices to study local convergence properties of iterative algorithms
because suitable initial conditions can be calculated offline.

The concept of an approximate zero was introduced in~\cite{bk:Blum:newt}.
An equivalent concept will be used here, however subsequent
developments differ. In~\cite{bk:Blum:newt}, attention was restricted
to analytic functions and global constants were sought for use in
one-time algorithms (for solving polynomial equations), as opposed
to the focus here on real-time optimisation algorithms.

The Newton iteration for finding a critical point of $h: \reals^n
\rightarrow \reals$ is $x_{k+1} = x_k - [h''(x_k)]^{-1} h'(x_k)$.
Its invariance to affine changes of coordinates means it suffices
to assume in this section that the critical point of interest is
located at the origin.  The Euclidean norm and Euclidean inner
product are used throughout for $\reals^n$.

\begin{definition}[Approximate Critical Point]
Let $h: \reals^n \rightarrow \reals$ be a smooth function with a
non-degenerate critical point at the origin: $Dh(0) = 0$ and $D^2h(0)$
is non-singular.  A point $x$ is an \emph{approximate critical
point} if, when started at $x_0 = x$, the Newton iterates $x_k$ at
least double in accuracy per iteration: $\| x_{k+1} \| \leq \frac12
\| x_k \|$.
\end{definition}

Provided the critical point is non-degenerate, the set of approximate
critical points contains a neighbourhood of the critical point.
For the development of homotopy-based algorithms in Section~\ref{sec:halg},
it is desirable to have techniques for finding a $\rho > 0$ such
that all points within $\rho$ of the critical point are approximate
critical points.  Two techniques will be explored, starting with
the one-dimensional case for simplicity.

\begin{example}
\label{ex:cubic}
Let $h(x) = x^2 + x^3$.  Then $h'(x) = 2x+3x^2$ and $h''(x) = 2+6x$.
The Newton iterate is $x \mapsto x - \frac{2x+3x^2}{2+6x} =
\frac{3x^2}{2+6x}$.  Graphing this function shows that the largest
interval $[-\rho,\rho]$ containing only approximate critical points
is constrained by the equation $\frac{3x^2}{2+6x} = -\frac{x}2$ for
$x<0$.  In particular, $\rho = \frac16 \approx 0.17$ is the best possible.
\end{example}

Explicit calculation as in Example~\ref{ex:cubic} is generally not
practical.  It will be assumed that on an interval $I$ containing
the origin the first few derivatives of $h$ are bounded.
Since $h'(0)=0$, a basic approximation for $h'(x)$ on $I$ is $h'(x)
= x h''(\bar x)$ for some $\bar x \in I$.  It follows that if
$h''(\bar x) / h''(x)$ is bounded between $\frac12$ and $\frac32$
for $x, \bar x \in I$ then all points in $I$ are approximate critical
points.  Moreover, $h'''(x)$ can be used to bound the change in
$h''(x)$.  This makes plausible the following lemma.

\begin{lemma}
\label{lem:sbound}
Let $h: \reals \rightarrow \reals$ be a smooth function with $h'(0)=0$
and $h''(0) \neq 0$.  Let $I$ be an interval containing the origin
and $\alpha = \sup_{x \in I} |h'''(x)|$.  Let $\rho =
\frac{|h''(0)|}{2\alpha}$.  Then every point in the interval
$[-\rho,\rho] \cap I$ is an approximate critical point of $h$.
\end{lemma}
\begin{proof}
Follows from Proposition~\ref{pr:acp} upon noting that
$h''(x) - h''(y) = h'''(\bar x) (x-y)$ for some $\bar x$ lying
between $x$ and $y$.
\end{proof}

\begin{example}
In Example~\ref{ex:cubic}, $h'''(x) = 6$.  Applying Lemma~\ref{lem:sbound}
gives $\rho = \frac16 \approx 0.17$, coincidentally agreeing with
the best possible bound.
\end{example}

The second technique is to look at the derivative of the Newton map
$x \mapsto x - \frac{h'(x)}{h''(x)}$, which is $\frac{h'(x)
h'''(x)}{[h''(x)]^2}$.  Provided the magnitude of this derivative
does not exceed $\frac12$ then $x$ is an approximate critical point.

\begin{example}
In Example~\ref{ex:cubic}, $\frac{h'(x) h'''(x)}{[h''(x)]^2} =
\frac{3x(2+3x)}{(1+3x)^2}$.  Its magnitude does not exceed $\frac12$
provided $|x| \leq \frac{3-\sqrt{6}}9 \approx 0.06$. 
\end{example}

The need for evaluating $\frac{h'(x) h'''(x)}{[h''(x)]^2}$
can be avoided by using bounds on derivatives; an upper bound
on $|h'''(x)|$ gives a lower bound, linear in $x$, on $h''(x)$,
and an upper bound, quadratic in $x$, on $|h'(x)|$.
Nevertheless, the first technique appears to be preferable,
and will be the one considered further.

\begin{lemma}
\label{lem:barH}
Let $h: \reals^n \rightarrow \reals$ have a
non-degenerate critical point at the origin.  Let $H_x \in \reals^{n
\times n}$, a symmetric matrix, denote its Hessian at $x$, that is,
$D^2h(x) \cdot (\xi,\xi) = \langle H_x \xi, \xi \rangle$.
Let $\bar H_x$ denote the averaged Hessian $\bar H_x = \int_0^1
H_{tx}\,dt$.  Then $x$ is an approximate critical point if 
$\| H_x^{-1} \bar H_x - I \| \leq \frac12$, where the norm is
the operator norm.
\end{lemma}
\begin{proof}
The gradient of $h$ at $x$ is $\int_0^1 H_{tx} x\,dt = \bar H_x x$.  Therefore
the Newton map is $x \mapsto x - H_x^{-1} \bar H_x x$.  If
$\| H_x^{-1} \bar H_x - I \| \leq \frac12$ then $\| x - H_x^{-1} \bar H_x x \|
\leq \frac12 \|x\|$, as claimed.
\end{proof}

\begin{lemma}
\label{lem:nbound}
With notation as in Lemma~\ref{lem:barH}, if
$\|H_x - H_0\| < \|H_0^{-1}\|^{-1}$ then
\begin{equation}
\| H_x^{-1} \bar H_x - I \| \leq \frac{\|\bar H_x - H_x\|}
{\|H_0^{-1}\|^{-1} - \|H_x-H_0\|}.
\end{equation}
\end{lemma}
\begin{proof}
Let $A = -(H_x-H_0)H_0^{-1}$. Then $\|A\| \leq \|H_x-H_0\| \|H_0^{-1}\|
< 1$.  Therefore $\|(I-A)^{-1}\| = \|I+A+A^2+\cdots\| \leq 1 + \|A\|
+ \|A\|^2 + \cdots = (1-\|A\|)^{-1}$. Moreover,
$\|H_x^{-1} \bar H_x - I\| = \|H_0^{-1} (I-A)^{-1} (\bar H_x -
H_x)\| \leq \|H_0^{-1}\| (1-\|A\|)^{-1} \|\bar H_x - H_x\|$.
Finally, note $(1-\|A\|)^{-1} \leq (1-\|H_x-H_0\|\|H_0^{-1}\|)^{-1}$.
\end{proof}

A bound on the third-order derivative yields a Lipschitz constant
for the Hessian.

\begin{proposition}
\label{pr:acp}
Define $h$ and $H_x$ as in Lemma~\ref{lem:barH}.  Let $I$ be a
star-shaped region about the origin.  Let $\alpha \in \reals$ be
such that $\|H_x - H_y\| \leq \alpha \|x-y\|$ for $x,y \in I$.  Let
$\rho = (2\alpha \|H_0^{-1}\|)^{-1}$.  If $x \in I$ and $\|x\| \leq
\rho$ then $x$ is an approximate critical point.
\end{proposition}
\begin{proof}
First, $\|\bar H_x - H_x\| \leq \int_0^1 \|H_{tx}-H_x\|\,dt \leq
\alpha \|x\| \int_0^1 1-t\,dt = \frac\alpha2\|x\|$.
Also, $\|H_x - H_0\| \leq \alpha \|x\| \leq \frac12 \|H_0^{-1}\|^{-1}$.  
Lemma~\ref{lem:nbound} implies
$\| H_x^{-1} \bar H_x - I \| \leq \frac{(4\|H_0^{-1}\|)^{-1}}
{\|H_0^{-1}\|^{-1} - (2\|H_0^{-1}\|)^{-1}}$.
The result now follows from Lemma~\ref{lem:barH}.
\end{proof}

\section{A Homotopy-based Algorithm for Optimisation}
\label{sec:halg}

This section outlines how a homotopy-based algorithm can solve
fibre bundle optimisation problems efficiently.

Homotopy-based algorithms have a long history~\cite{bk:Allgower:num_cont}.
Attention has mainly focused on one-time problems where little use
can be made of results such as Proposition~\ref{pr:acp} requiring
the prior calculation of various bounds on derivatives and locations
of critical points.  Time spent on prior calculations is better
spent on solving the one-time problem directly.  The reverse is
true for real-time algorithms.  The more calculations performed
offline, the more efficient the real-time algorithm can be made,
up until when onboard memory becomes a limiting factor.

Definition~\ref{def:nice} may make it appear that nice optimisation
problems are not necessarily that nice if the sets $\Theta_i$ are
complicated.  However, it is always straightforward to find
fibre-wise critical points by path following.  The worst that can
happen if the $\Theta_i$ are complicated is that the algorithm may
need to follow more than one path because it cannot be sure which
path contains the sought after global minimum.

\begin{proposition}
\label{pr:track}
With notation as in Lemma~\ref{lem:N}, let $\gamma: [0,1] \rightarrow
\Theta$ be a smooth path.  Let $p \in N \cap \pi^{-1}(\gamma(0))$.
Then $\gamma$ lifts to a unique smooth path $\tilde\gamma: [0,1]
\rightarrow N$ such that $\tilde\gamma(0)=p$ and $\pi(\tilde\gamma(t))
= \gamma(t)$ for $t \in [0,1]$.
\end{proposition}
\begin{proof}
Follows from Lemma~\ref{lem:N} in a similar way
Proposition~\ref{prop:nice} did.
\end{proof}

\begin{corollary}
With notation as in Lemma~\ref{lem:N}, the number of points in the
set $N \cap \pi^{-1}(\theta)$ is constant for all $\theta \in
\Theta$.
\end{corollary}

Different paths with the same end points can have different lifts.
Nevertheless, as the number of fibre-wise critical points is constant
per fibre, as soon as the fibre-wise critical points on one fibre
are known, the fibre-wise critical points on any other fibre can
be found by following any path from one fibre to another.  Furthermore,
only paths containing local minima need be followed to find a global
minimum.

\begin{proposition}
\label{pr:type}
With notation as in Lemma~\ref{lem:N}, let $p$ and $q$ lie on a
connected component of $N$.  Then $p$ is a fibre-wise local minimum
if and only if $q$ is a fibre-wise local minimum.  
\end{proposition}
\begin{proof}
Fibre-wise, each critical point is assumed non-degenerate.  Therefore,
along a continuous path, the eigenvalues of the Hessian cannot
change sign and the index is preserved.
\end{proof}

Referring to Proposition~\ref{pr:type}, define $\tilde N \subset
N$ to be the connected components of $N$ corresponding to fibre-wise
local minima.

An outline of a homotopy-based algorithm for fibre bundle optimisation
problems can now be sketched.  It will be refined presently.  It
relies on several lookup tables, the first of which has entries
$(\theta, \pi^{-1}(\theta) \cap \tilde N)$ for $\theta \in
\{\theta_1,\cdots,\theta_n\} \subset \Theta$.  That is to say, the
set of all local minima of $f$ restricted to the fibres over
$\theta_1,\cdots,\theta_n$, have been determined in advance.

\begin{enumerate}
\item Given $\theta$ as input, determine an appropriate starting
point $\theta_i$ from the finite set $\{\theta_1,\cdots,\theta_n\}$.

\item Determine an appropriate path $\gamma$ from $\theta_i$ to $\theta$.

\item Track each fibre-wise critical point $p \in \pi^{-1}(\theta_i)
\cap \tilde N$ along the path $\gamma$ (i.e., numerically compute
the lift $\tilde\gamma$ defined in Proposition~\ref{pr:track}).

\item Evaluate the cost function $f$ at the fibre-wise local minima on
the fibre $\pi^{-1}(\theta)$ to determine which are global minima.
Return one or all of the global minima.
\end{enumerate}

Step 3 can be accomplished with a standard path-following
scheme~\cite{bk:Allgower:num_cont}. A refinement is to utilise
Proposition~\ref{pr:acp}, as now explained.  Using a suitably chosen
local coordinate chart, the cost function $f$ restricted to a
sufficiently small segment of the path $\gamma$ can be represented
locally by a function $h: \reals^n \times \reals \rightarrow \reals$.
Here, $h$ should be thought of as a parametrised cost function,
with $h(\cdot;0)$ the starting function having a non-degenerate
critical point at the origin, and the objective being to track that
critical point all the way to the cost function $h(\cdot;1)$.  An
\emph{a priori} bound on the location of the critical point of
$h(\cdot;t)$ is readily available; see for example~\cite[Chapter
16]{bk:Hirsch:diff}.  Similarly, Proposition~\ref{pr:acp} gives a
bound on how far away from the critical point the initial point can
be whilst ensuring the Newton method converges rapidly.  Therefore,
these two bounds enable the determination of the largest value of
$t \in [0,1]$ such that, starting at the origin, the Newton method
is guaranteed to converge rapidly to the critical point of $h(\cdot;t)$.
Once that critical point has been found, a new local chart can be
chosen and the process repeated.

These same bounds, which are pre-computed and stored in lookup
tables, permit the determination of the number of Newton steps
required to get sufficiently close to the critical point. For
intermediate points along the path, it is not necessary for the
critical points to be found accurately. Provided the algorithm
stays within the bound determined by Proposition~\ref{pr:acp}, the
correct path is guaranteed of being followed.

The fact that $M$ may be a manifold presents no conceptual difficulty.
As in~\cite{Manton:opt_mfold}, it suffices to work in local
coordinates, and change charts as necessary, as already mentioned
earlier.

Steps 1 and 2 of the algorithm pose three questions.  How should the
set $\{\theta_1,\cdots,\theta_n\}$ be chosen, how should a particular
$\theta_i$ be selected based on $\theta$, and what path should be
chosen for moving from $\theta_i$ to $\theta$?  Importantly, the algorithm
will work regardless of what choices are made.  Nevertheless,
expedient choices can significantly enhance the efficiency of the algorithm.

Another refinement is to limit in Step 3 the number of paths that are
followed. Proposition~\ref{prop:nice} ensures that it is theoretically
possible to determine beforehand which path the global minimum 
will lie on.  Therefore, with the use of another lookup table,
the number of paths the algorithm must track can be reduced;
see Remark~\ref{rem:theta}.

\section{Conclusion}

A nascent theory of optimisation geometry was propounded for studying
real-time optimisation problems.  It was demonstrated that irrespective
of how difficult an individual cost function might be to optimise
offline, a simple and reliable homotopy-based algorithm can be used
for the real-time implementation.

Real-time optimisation problems were reformulated as fibre bundle
optimisation problems (Definition~\ref{def:fbop}).  The geometry
inherent in this fibre bundle formulation provides information about
the problem's intrinsic computational complexity. An advantage of
studying the geometry is it prevents any particular choice of
coordinates from dominating, so there is a possibility of seeing
through obfuscations caused by the chosen formulation of the problem.

That geometry helps reveal the true complexity of an optimisation
problem can be demonstrated by referring back to the discussion of
the fibre bundle optimisation problem on the torus in
Section~\ref{sec:torus}.  Irrespective of how complicated the
individual cost functions are (but with the proviso that they be
fibre-wise Morse), the fibre-wise critical points will lie on a
finite number of circles that wind around the torus, and because
these circles cannot cross each other, or become tangent to a fibre,
they each wind around the torus the same number of times.  Therefore,
in terms of where the fibre-wise critical points lie, the intrinsic
complexity is encoded by just two integers: the number of circles,
and the number of times each circle intersects a fibre.

Although this article lacked the opportunity to explore this aspect,
a crucial observation is even though it may appear that some
problems are more complicated than others because the paths of
fibre-wise critical points locally ``fluctuate'' more, a smooth
transformation can be applied to iron out these fluctuations.
Smooth transformations
cannot change the \emph{intrinsic complexity} whereas they can, by
definition, eliminate \emph{extrinsic complexity}.

The second determining aspect of complexity is the number
of times the fibre-wise minimum jumps from one circle to another.
This is precisely what is counted by the topological complexity,
mentioned just after Definition~\ref{def:nice}.

For higher dimensional problems, attention can always be restricted
to compact one-dimensional submanifolds of the parameter space
$\Theta$, in which case the situation is essentially the same as
for the torus; see Remark~\ref{rem:restrict}.  The only difference
is the circles may become intertwined.  The theory of links and
braids may play a role in further investigations, for if two circles
are linked then no smooth transformation can separate them.

Another potentially interesting direction for further work is to
explore the possibility of replacing a family of cost functions
with an equivalent family which is computationally simpler to work
with but which gives the same answer.

There are myriad other opportunities for refinements and extensions.
The theory presented in this article was the first that came to
mind and may well be far from optimal.

\section*{Acknowledgments}

Not only has Uwe brought happiness into my personal life with his
genuine friendship and good humour, Uwe has played a pivotal role
in my academic life.  It is with all the more pleasure and sincerity
then that I dedicate this article to Uwe Helmke on the occasion of
his 60th birthday.

This work was supported in part by the Australian Research Council.

}